\documentclass[11pt,reqno]{amsart}
\usepackage[margin=0.95in]{geometry}
\usepackage{amssymb,amsfonts,amsmath,mathrsfs,etoolbox,mathtools,bm}
\usepackage[noadjust]{cite}
\usepackage[dvipsnames]{xcolor}
\usepackage[breaklinks=true, linktocpage=true,
colorlinks=true,linkcolor=teal,citecolor=teal,filecolor=RoyalBlue,urlcolor=BlueViolet]{hyperref}
\hypersetup{linktocpage}
\usepackage[capitalise]{cleveref}
\usepackage[inline]{enumitem}  
\usepackage{mathtools}
\usepackage{comment}
\usepackage{soul}
\usepackage{tikz}
\usepackage{tikz-cd}
\usetikzlibrary{shapes.geometric}
\usetikzlibrary{shapes.misc}
\usetikzlibrary{positioning}
\allowdisplaybreaks[4]
\usepackage{hypbmsec}
\newtheorem{theorem}{Theorem}[section]

\newtheorem{proposition}[theorem]{Proposition}
\usepackage[utf8]{inputenc}
\newtheorem{lemma}[theorem]{Lemma}

\newtheorem{conjecture}[theorem]{Conjecture}

\newtheorem{remark}[theorem]{Remark}
\newtheorem{definition}[theorem]{Definition}
\newtheorem{question*}{Question}

\theoremstyle{definition}

\newtheorem{example}[theorem]{Example}

\newtheorem{notation}[theorem]{Notation}

\numberwithin{equation}{section}

\newcommand{\Z}{\mathbb{Z}}
\newcommand{\R}{\mathbb{R}}
\newcommand{\Q}{\mathbb{Q}}

\renewcommand{\tilde}{\widetilde}

\renewcommand{\bar}{\overline}

\renewcommand{\Im}{\mathrm{Im}}

\renewcommand{\pmod}[1]{\, (\mathrm{mod} {\, #1})}

\renewcommand{\Pr}{\mathbb{P}}

\renewcommand{\Re}{\mathrm{Re}}

\usepackage{tikz}
\usetikzlibrary{calc}

\newcommand{\cV}{\mathcal{V}}

\patchcmd{\section}{\scshape}{\bfseries}{}{}
\makeatletter
\renewcommand{\@secnumfont}{\bfseries}
\makeatother

\makeatletter\newcommand{\tpmod}[1]{{\@displayfalse\pmod{#1}}}
\makeatletter
\@namedef{subjclassname@2020}{\textup{2020} Mathematics Subject Classification}
\makeatother

\begin{document}

\title{Real Geometric transcendence for the Gamma function}

\author{Arshay Sheth}
\address[Arshay Sheth]{School of Mathematics, Tata Institute of Fundamental Research, Homi Bhabha Road,
Mumbai - 400005, India.}
\urladdr{}

\thanks{}

\email{asheth@math.tifr.res.in}

\begin{abstract}
We show that the  $x$-axis is the only real algebraic curve in $\mathbb R^2$ whose image via the Gamma function is contained in an algebraic curve. Our proof employs an elegant base-change argument due to Tamiozzo (2023) to  deduce the result from the corresponding complex geometric transcendence result of Eterovi\'c, Padgett and Zhao (2025). As an application, we use the complex and real geometric transcendence results to study analogues of the Manin--Mumford conjecture for the Gamma function. 
\end{abstract}

\maketitle

\section{Introduction} \label{intro}

If $\Omega: X \rightarrow Y$ is a transcendental map between two complex algebraic varieties, the image of a generic algebraic subvariety of $X$ will usually not be an algebraic subvariety of $Y$; it is thus often of special interest to characterise the relevant \textit{bialgebraic varieties} for the map $\Omega$, \textit{i.e.},  those varieties $V \subseteq X$ such that $\Omega(V)$ is also algebraic. 
For instance, letting $\exp(z):=e^{2 \pi i z}$, the Ax--Lindemann--Weierstrass theorem asserts that irreducible algebraic subvarieties of $\mathbb C^n$ whose image via the map 
\begin{align*}
\mathbb C^n & \rightarrow (\mathbb C^\times)^n \\
(z_1, \ldots, z_n) & \mapsto (\exp(z_1), \ldots ,\exp(z_n))
\end{align*}
is algebraic, are precisely translates of linear subspaces of $\mathbb C^n$ defined over $\mathbb Q$. Similar results are known, for instance, for the modular $j$-function and for Weierstrass elliptic functions. 
The problem of determining the relevant bialgebraic varieties in a given situation 
can be thought of as a problem of  understanding how the algebraic structures on both the domain and the codomain interact under the map $\Omega$; this problem can 
also be regarded as a geometric analogue of a fundamental problem in transcendental number theory where, given a holomorphic transcendental function $f$, one is interested in determining all the \textit{bialgebraic numbers} with respect to $f$,  \textit{i.e.}, all algebraic numbers whose values under $f$ remain algebraic. 

\subsection{Complex geometric transcendence for the $\Gamma$-function}

The focus of the present paper is to explore the geometric transcendence properties of the celebrated Gamma function, which is defined by 
$\begin{displaystyle}
\Gamma(z)= \int_{0}^{\infty} e^{-t} t^{z-1} dt    \hspace{4mm} \text{ for } \Re(z)>0 \end{displaystyle}$ and admits a meromorphic continuation to the entire complex plane with simple poles at the non-positive integers $\mathbb Z_{\leq 0}$. 

\begin{theorem}[Eterovi\'c--Padgett--Zhao \cite{EPZ25}] \label{complexgt}

Consider the function
\begin{align*}
\Gamma^2: (\mathbb C  \smallsetminus \mathbb Z_{\leq 0})^{2} & \longrightarrow \mathbb C^2 \\ 
(z_1, z_2) & \mapsto (\Gamma(z_1), \Gamma(z_2)).
\end{align*}
Suppose $C \subseteq \mathbb C^2$ is an irreducible algebraic curve such that $\Gamma^2(C \cap (\mathbb C\smallsetminus \mathbb Z_{\leq 0})^2)$ is contained in an algebraic curve. Then $C$ is defined by one of the following equations:
\begin{enumerate}
\item $X=Y$;
\item  $X=w$ for some $w \in \mathbb C$; 
\item $Y=w$ for some $w \in \mathbb C$. 
\end{enumerate}
\end{theorem}

The situation of geometric transcendence for the $\Gamma$ function differs from those of the exponential function, the $j$-function and the Weierstrass elliptic functions, since unlike the latter three functions, $\Gamma$ does not satisfy an algebraic differential equation by H\"older's theorem. Indeed, the proof of Theorem \ref{complexgt} does not use o-minimal methods, but rather proceeds by complex analytic computations describing the
behavior of the fibers of $\Gamma$. We refer to \cite{DVP25, EP25} for further progress on the topic of geometric and functional transcendence for the Gamma function.

\begin{remark}
The varieties appearing in Theorem \ref{complexgt} (and their natural generalisations to higher dimensions) are called \textit{trivially bialgebraic}  in the notation of \cite{EPZ25}; indeed, they are always bialgebraic for the $n$-fold product of any set-theoretic function. Theorem \ref{complexgt} is a special case of the main theorem of \cite{EPZ25}, which deals with the map  
\begin{align*}
\Gamma^n: (\mathbb C  \smallsetminus \mathbb Z_{\leq 0})^{n} & \longrightarrow \mathbb C^n \\ 
(z_1, \ldots ,z_n) & \mapsto (\Gamma(z_1), \ldots,  \Gamma(z_n))
\end{align*}
for any $n \geq 2$. If $V \subseteq \mathbb C^n$ is an irreducible algebraic subvariety such that the dimension of the Zariski closure of $\Gamma^n(V \cap (\mathbb C  \smallsetminus \mathbb Z_{\leq 0})^{n})$ equals the dimension of $V$,  the authors of \cite{EPZ25} show that $V$ is trivially bialgebraic. 
\end{remark}

\subsection{Real geometric transcendence} 
One can investigate  analogues of all of the above results in the setting of real algebraic geometry, \textit{i.e.}, by regarding the relevant spaces as real algebraic varieties. The study of real analogues of complex geometric transcendence results was initiated by Tamiozzo in \cite{Tam23}, where he proved real geometric transcendence results for the exponential and the $j$-function. The results obtained in these cases have interesting connections to various arithmetic objects, such as class numbers of real quadratic fields and special geodesics in the upper half-plane.  The techniques of \cite{Tam23} were extended by the author and Tamiozzo in \cite{ST25} to deal with the case of the Weierstrass elliptic functions and, in current work in progress, to deal with  uniformisation maps of higher genus curves.  In this paper, we prove the real analogue of Theorem \ref{complexgt}.  Let $S$ denote the subset $\mathbb Z_{\leq 0} \times \{0\}$ of $\R^2$;  by making the natural identification $\mathbb R^2 \simeq \mathbb C$, we are led to consider the function
\begin{align*}
    G : \mathbb R^2  \smallsetminus S  & \rightarrow \R^2\\
    (x, y) & \mapsto (\Re (\Gamma(x+iy)), \Im (\Gamma(x+iy))).
\end{align*}

While studying geometric transcendence results in the real setting, it is often fruitful to weaken the bialgebraicity property by requiring the relevant images of an algebraic subvariety in the domain to be contained in (as opposed to being equal to) an algebraic subvariety in the codomain; this intuition is formalised in the following definition. 

\begin{definition}\label{bialgd}
A non-empty subset $\cV \subsetneq \mathbb R^2$ is called a weakly bialgebraic curve for $G$ if 
\begin{enumerate}
\item $\cV$ is not a singleton and there exist algebraic subvarieties $V \subset \mathbb A^2_\mathbb R$ and $W \subsetneq \mathbb A^2_\mathbb R$ such that $\cV=V(\mathbb R)$ and  $G(\cV \cap (  \mathbb R^2  \smallsetminus S)) \subset W(\mathbb R)$;
\item the set $\cV$ cannot be written in the form $V_1(\mathbb R) \cup V_2(\mathbb R)$, where $V_1, V_2 \subset \mathbb A^2_\mathbb R$ are algebraic subvarieties and the inclusions $V_i(\mathbb R)\subset \cV$ are proper for $i=1, 2$. 
\end{enumerate}
\end{definition}

\begin{theorem}\label{mainthm}
Suppose $\mathcal V \subset \mathbb R^2$ is a weakly bialgebraic curve for $G$. Then $\mathcal V$ must be the $x$-axis. 
\end{theorem}

The fact that the $x$-axis is indeed weakly bialgebraic follows from the relation
\begin{equation} \label{conj}
\bar{\Gamma(z)}= \Gamma(\bar{z}),  
\end{equation}
and we see that the image of the $x$-axis under $G$ is contained in the $x$-axis.  In Section \ref{mm}, we use Theorem \ref{complexgt} and Theorem \ref{mainthm} to study complex and real analogues of the Manin--Mumford conjecture in this setting.  To the best of our knowledge, this is the first instance in the literature where such analogues have been studied for the Gamma function. 

\subsection*{Acknowledgements} This article is an outgrowth of the author's work \cite{ST25}  with Matteo Tamiozzo. I would like to thank him for introducing me to the topic of geometric transcendence and for many helpful discussions.  I am also grateful to Sebastian Eterovi\'c, Adele Padgett, and Roy Zhao for helpful comments on a previous version of this article.

\section{Proof of Theorem \ref{mainthm}}

\subsection{Preliminaries}

Let $\mathcal V$ be a weakly bialgebraic curve for $G$. By considering sums of squares of polynomials, we can write $\mathcal V$ as the zero locus of a single polynomial $R \in \mathbb R[X, Y]$.  By Definition \ref{bialgd} (2), the set $\mathcal V$ must be the vanishing locus of some irreducible factor $P$ of $R$. 
Furthermore, since $\cV$ is infinite, the polynomial $P$ is also irreducible in $\mathbb C[X, Y]$ (otherwise, we could write $P=P_1 \bar{P_1}$ for an irreducible polynomial $P_1 \in \mathbb C[X, Y]$ that is not a scalar multiple of a real polynomial, and both $P_1$ and $\bar{P_1}$ would vanish on $\cV$, contradicting Bézout's theorem).

\subsection{The base-change diagram}

By Equation \eqref{conj}, 
$$
\Re (\Gamma(x+iy)) = \frac{\Gamma(x+iy)+\Gamma(x-iy)}{2}, \hspace{4mm} \Im (\Gamma(x+iy))=  \frac{\Gamma(x+iy)-\Gamma(x-iy)}{2i}. 
$$

We define maps 
\begin{equation*}
\begin{aligned}[t]
  f\colon & \mathbb C^2  \rightarrow \mathbb C^2 \\ & (v, w) \mapsto (v+iw, v-iw)
\end{aligned}
\quad \hspace{5mm} \quad
\begin{aligned}[t]
  g \colon & \mathbb C^2  \rightarrow \mathbb C^2 \\
 & (a, b) \mapsto \left ( \frac{a+b}{2}, \frac{a-b}{2i} \right ),
\end{aligned}
\end{equation*}

and thus obtain a commutative diagram 

\begin{center}
\begin{equation} \label{commdiagram}
\begin{tikzcd}
 \mathbb R^2  \smallsetminus S \arrow[r, "G"] \arrow[hookrightarrow, d] &[0.5em] \mathbb R^2 \arrow[hookrightarrow, r] & \mathbb C^2 \arrow[d, "\iota \circ g^{-1}"]\\
\mathbb C^2 \arrow[r, "f"]  & \mathbb C^2 \arrow[r, "\Gamma^2"] & \mathbb \Pr^1(\mathbb C)^2, 
\end{tikzcd}
\end{equation}
\end{center}

where the left vertical and the top right horizontal maps above are induced by the inclusion $\mathbb R \subseteq \mathbb C$, the map $\iota: \mathbb C^2 \rightarrow \mathbb P^1(\mathbb C)^2$ is given by the inclusion $\mathbb C \subseteq  \mathbb P^1(\mathbb C)$ on each component and where we let $\Gamma^2$ denote, using the same name, the natural extension of the map in Theorem \ref{complexgt}.  
Let $\tilde{\cV}=\cV \cap  (\mathbb R^2  \smallsetminus S)$. Since $\cV$ is weakly bialgebraic, there exists a non-zero polynomial $Q \in \mathbb R[X, Y]$ such that $G(\tilde{\cV})$ is contained in the real algebraic curve with equation $Q=0$. Let $C_Q$ be the complex plane curve with equation $Q=0$, and let $\hat{C}_Q\subset \Pr^1(\mathbb C)^2$ be the Zariski closure of $\iota \circ g^{-1}(C_Q)$. Let $C_P$ be the complex curve with equation $P=0$ and let $q=\Gamma^2 \circ f$. The commutativity of the diagram \eqref{commdiagram} implies that $A:=C_P \cap q^{-1}(\hat{C}_Q)$ contains $\tilde{\cV}$.

\subsection{Analytic continuation from the complexification of a real algebraic curve}
Viewing $C_P$ as a complex analytic space with respect to the Euclidean topology, we thus see that $A$ is an analytic subset of $C_P$.   
Using the fact that $A$ contains $\tilde{\cV}$, and hence segments of the real algebraic curve $\mathcal V$, we show below in Proposition \ref{lem:realcompbalg} that this condition is rigid enough to yield that $f(C_P)$ is bialgebraic for $\Gamma^2$, \textit{i.e.}, $f(C_P)$ satisfies the hypothesis in Theorem \ref{complexgt}. We shall employ the formalism of thin sets and a version of the identity principle from \cite{GR84}.

We recall (cf. \cite[page 132]{GR84}) that a closed subset $A$ of a complex analytic space $X$ is called \textit{thin} if every $p \in A$ has an open neighbourhood $U$ such that $A \cap U$ is contained in a nowhere dense analytic subset of $X$.
In particular, we will use in Proposition \ref{lem:realcompbalg} below that if a subset of a complex curve contains a non-empty open set (in the complex analytic topology), then it is not thin.   
We will also use  ~\cite[Theorem, page 168]{GR84}  
which implies that if $X$ is an irreducible complex plane curve, then every proper analytic set of $X$ is thin in $X$.

\begin{proposition} \label{lem:realcompbalg}
The curve $f(C_P)$ is bialgebraic for $\Gamma^2$.
\end{proposition}
\begin{proof}
 Since $\mathcal V$ is infinite, $\tilde{\cV}$ contains a subset $I$ which is homeomorphic to an open interval.  Since any complex curve has only finitely many singular points, we may choose a smooth point $x$ of $C_P$ lying in $I$.  By the implicit function theorem, there exists an open subset $U \subseteq C_P$ (in the complex analytic topology) containing $x$, an open disc $D \subseteq \mathbb C$,  and a biholomorphic map $\phi: D \rightarrow U$.  By shrinking $U$ if necessary we  may also assume that  $f(U) \subseteq (\mathbb C  \smallsetminus \mathbb Z_{\leq 0})^{2}$. We now note that $(Q \circ g \circ q) \circ \phi$ vanishes on $\phi^{-1}(I \cap U)$; by the identity theorem, this function must vanish on $D$ as well.  So $Q \circ g \circ q$ vanishes on $U$ and hence $U \subseteq A$. 
This implies that the analytic subset $A$ of $C_P$ is not thin,  
and so we must have $A=C_P$. Therefore $q(C_P)\subseteq \hat{C}_Q$ and so $\Gamma^2(f(C_P)) \subseteq \hat{C}_Q $ which in turn implies that $\Gamma^2(f(C_P) \cap (\mathbb C  \smallsetminus \mathbb Z_{\leq 0})^{2} ) \subseteq g^{-1}(C_Q)$, proving that  $f(C_P)$ is bialgebraic for $\Gamma^2$. 
\end{proof}

\subsection{Application of Theorem \ref{complexgt}}

By Theorem \ref{complexgt} and Proposition \ref{lem:realcompbalg}, we conclude that $f(C_P)$ must equal a set of the form 
\begin{enumerate}
\item $\{(x, y) \in \mathbb C^2: x=c\}$ for some $c \in \mathbb C$.  

\item  $\{(x, y) \in \mathbb C^2: y=c\}$ for some $c \in \mathbb C$.  

\item  $\{(x, y) \in \mathbb C^2: x=y\}$. 
\end{enumerate}

Since $f(\mathcal V) \subseteq f(C_P)$, we deduce that $f(\mathcal V)$ is contained in a set of the form described above.

\subsection{Conclusion} We note that $f(\mathcal V)$ cannot be contained in a horizontal or vertical line; indeed, in that case there would be a $c =a+i b \in \mathbb C$ such that  every $(x, y) \in \mathcal V$ satisfies $x \pm iy=a+ib$, contradicting the assumption that $\mathcal V$ is not a point. Thus, we must have that $f(\mathcal V) \subseteq \{(x, y) \in \mathbb C^2: x=y\}$, which implies that $x+iy=x-iy$ for all $(x, y) \in \mathcal V$. This forces $y=0$ and so $\mathcal V$ must be contained in, and hence equal to the $x$-axis by Bézout's theorem. 
Conversely, using Equation \eqref{conj}, we see that the image of the $x$-axis is contained in the $x$-axis, so the $x$-axis is indeed weakly bialgebraic. This completes the proof of Theorem \ref{mainthm}.

\section{The image of the $x$-axis under $G$}

\begin{proposition} \label{complete}
Let $\mathcal V$ denote the $x$-axis in $\mathbb R^2$. We have that 
$$G(\cV \cap ( \mathbb R^2  \smallsetminus S)) = \{ (x, 0) \in \mathbb R^2: x \in \mathbb R,  x \neq 0\}.$$
\end{proposition}

\begin{proof}
It suffices to show that $\Gamma(\mathbb R \smallsetminus \mathbb Z_{\leq 0})=\mathbb R \smallsetminus \{0\} $. 
Since $\Gamma$ has a global minimum on the interval $(0, \infty)$ where it achieves the value $\approx 0.88 $, we instead study its behaviour on the negative real axis to prove Proposition \ref{complete}. 
We recall the reflection formula $\Gamma(x)=\frac{\pi}{\sin(\pi x)\Gamma(1-x)}$, which is valid for all $x \not \in \mathbb Z$. It follows that for each $n \geq 0$, $\Gamma(x)$ tends to $-\infty$ (resp. $+\infty$)  as $x$ approaches either of the end-points of the interval $(-2n-1, -2n)$ (resp. $(-2n-2, -2n-1))$.  Let $M_n$ denote the global maximum  of $\Gamma(x)$ on $(-2n-1, -2n)$ and $m_n$ denote the global minimum of $\Gamma(x)$ on  $(-2n-2, -2n-1)$; from the reflection formula and the fact that $\Gamma(x) \to  \infty$ as $x \to + \infty$, it follows that $M_n \to 0^{-}$ and $m_n \to 0^+$ as $n \to \infty$. Since $\Gamma$ is continuous on each of the above intervals,  we conclude that the image of $\mathbb R_{\leq 0} \smallsetminus \mathbb Z_{\leq 0}$ under $\Gamma$ equals $\mathbb R \smallsetminus \{0\}$. 
\end{proof}

In particular, Proposition \ref{complete} implies that the image of the $x$-axis is a semialgebraic set. This is in harmony with the real geometric transcendence results of \cite{Tam23, ST25} where images of weakly bialgebraic sets were often also semialgebraic.

\begin{remark}
The algebraic rigidity of the x-axis under $G$ is in sharp contrast with the behaviour of the purely imaginary y-axis. Setting $z=it$ in the reflection formula for $t \in \R$, using Equation \eqref{conj} and the equation $\Gamma(z+1)=z\Gamma(z)$,  it follows that 
$
|\Gamma(it)|^2 = \frac{\pi}{-it \sin(\pi i t)} = \frac{\pi}{t \sinh(\pi t)}. 
$
Thus, $|\Gamma(it)|$ decays exponentially to zero as $t \to \infty$. By Stirling's approximation for $\log(\Gamma(z))$ (see for instance \cite[Proposition 9.6.27]{Coh07}),  it follows that $\Im(\log \Gamma (it))$ is asymptotic to $t \log t$ as $t \to \infty$. Since this imaginary part is unbounded, the image of the $y$-axis under $G$ intersects the $x$-axis infinitely many times as it spirals around, and converges exponentially to, the origin. Suppose that the image of the $y$-axis under $G$ is contained in the zero set of some non-zero $P \in \R[X, Y]$. Upon factoring $P=P_1 \cdots P_n$ into irreducible factors and applying the identity theorem, it follows that there exists $i \in \{1, \ldots, n\}$ such that $P_i$ vanishes on the image of the $y$-axis. Since this image also intersects the $x$-axis infinitely many times, Bézout's theorem implies that $P_i$ is a scalar multiple of $Y$; in other words, the image of the $y$-axis under $G$ should be contained in the $x$-axis, yielding a contradiction. Thus, the $y$-axis in $\mathbb R^2$ provides a concrete geometric illustration of a transcendental image under $G$.
\end{remark}

\section{Manin--Mumford analogues for $\Gamma$}
\label{mm}

\subsection{Recollection of general set-up} \label{gensetup} Let $f$ be a transcendental meromorphic function on $\mathbb C$ and let $S \subset \mathbb C$ be its set of poles. A point $a \in \mathbb C$ is called \textit{$f$-special} if $a \in \overline \Q$ and $a=f(b)$ for some $b \in \overline \Q$. A point $(a_1, a_2) \in \mathbb C^2$ is called \textit{f-special} if both $a_1$ and $a_2$ are $f$-special. If $\Omega_f: (\mathbb C \smallsetminus S)^2 \rightarrow \mathbb C^2$ is given by $(z_1, z_2) \mapsto (f(z_1), f(z_2))$, then an irreducible algebraic curve $W \subseteq \mathbb C^2$ is called $f$-special if there exists an algebraic curve $V \subseteq \mathbb C^2$ such that $V \cap (\mathbb C \smallsetminus S)^2$ is Zariski dense in $V$ and $\Omega_f(V \cap (\mathbb C \smallsetminus S)^2 ) \subseteq W$  . A \textit{Manin--Mumford type conjecture} in this setting predicts that if an algebraic curve $W \subseteq \mathbb C^2$ contains infinitely many $f$-special points, then $W$ must be $f$-special.

\begin{example}
If $f=\exp(z)$, then by the Gelfond--Schneider theorem, the $f$-special points are exactly the roots of unity in $\mathbb C$. It follows from the discussion in \S \ref{intro} that the $f$-special curves in $\mathbb C^2$ are precisely curves defined by  $X^m Y^n=\alpha$, where $m , n \in \mathbb Z$ and $(m, n) \neq (0, 0)$, and $\alpha \in \mathbb C^{\times}$. It was shown by Lang \cite{Lan65} (with independent proofs by Serre, Ihara and Tate) that if a curve in $\mathbb C^2$ contains infinitely many $f$-special points, then it must be $f$-special with the corresponding $\alpha$ being a root of unity.  
\end{example}

We refer to the survey article \cite{KUY18} for other examples, and generalisations, of Manin--Mumford type conjectures.

\subsection{Complex analogue for $\Gamma$} In view of the discussion in Section \ref{gensetup}, a point $a$ in $\mathbb C$ is called $\Gamma$-special if $a \in \overline{\mathbb Q}$ and if $a=\Gamma(b)$ for some $b \in \overline{\mathbb Q}$. Since $\Gamma(n)=(n-1)!$ for all $n \in \mathbb Z_{\geq 1}$, it follows that the factorials are $\Gamma$-special. As stated in \cite[page 239]{Riv12} in connection with the Lang--Rohrlich conjecture on the transcendence properties of the Gamma function, the converse is also believed to be true, \textit{i.e.}, the only algebraic numbers which are $\Gamma$-special are the factorials. 
We record this conjecture as follows. 

\begin{conjecture} \label{bialg}
Suppose $z \in \overline{\mathbb Q}$. Then $\Gamma(z) \in \overline{\mathbb Q}$ if and only if $z \in \Z_{\geq 1}$. 
\end{conjecture}

For example, $\Gamma(1/2)=\sqrt{\pi}$, and $\Gamma(1/3)$ and $\Gamma(1/4)$ are known to be transcendental by the work of Chudnovsky \cite{Chu76},  but the transcendence of $\Gamma(1/5)$ is still open.

We now proceed to study analogues of Manin--Mumford type conjectures for the Gamma function.  We first recall the  background about the asymptotic behaviour of algebraic curves from \cite[Section 3]{Wal92} that we shall use in our study.



\begin{remark}[Asymptotic behaviour of algebraic curves]\label{asymp}

Let $F(X, Y) \in \Z[X, Y]$ be an integer polynomial with positive degree in both $X$ and $Y$, which is irreducible in $\Q[X, Y]$. Write
$$
F(X, Y)= A_n(X) Y^n+A_{n-1}(X)Y^{n-1}+ \cdots +A_0(X). 
$$
Puiseux's theorem asserts the existence of $n$ distinct formal series 
$
Y_i(X) = \sum_{k=-f_i}^{\infty} c_{k, i} X^{-k/e_i}, 
$
 where each $e_i$ is a positive integer, each $f_i$ is an integer chosen such that $c_{-f_i, i} \neq 0$ and the $c_{k, i} \in \mathbb C$, such that 
$\displaystyle{
F(X, Y) = A_n(X) \prod_{i=1}^n (Y-Y_i(X))}$ as formal power series. As explained in \cite[page 162]{Wal92}, there exists  $R \in \mathbb R_{>0}$ such that each $Y_i(X)$ converges when $|X|>R$. It follows that there exists $R' \in \mathbb R_{>0}$  such that if  $(x, y) \in \mathbb C^2$ satisfies $F(x, y)=0$ with $|x|>R'$, then there exists $i \in \{1, \ldots, n\}$ such that $y=Y_i(x)$. This result describes the asymptotic behaviour of points $(x, y)$ lying on the curve $F=0$ as $x \to \infty$; for sufficiently large $x$, the points on the curve lie on graphs of finitely many fractional power series. 

\end{remark}

\begin{notation}
Let $(a_n)_n$ and $(b_n)_n$ be two sequences of positive real numbers. We say that $a_n \sim b_n$ if $\lim_{n \to \infty} \frac{a_n}{b_n}=1$, $a_n=O(b_n)$ if there exists a constant $C > 0$ such that $a_n \leqslant C b_n$ for all $n$ sufficiently large and $a_n= o(b_n)$ if  $\lim_{n \to \infty} \frac{a_n}{b_n}=0$. 
\end{notation}

\begin{proposition} \label{factorials}
Suppose $P \in \mathbb C[X, Y]$ is an irreducible polynomial such that it has infinitely many distinct roots of the form $(m!, n!)$ where $m, n \geqslant 1$. Then the curve $C_P \subseteq \mathbb C^2$ defined by $P=0$ must be a vertical line, a horizontal line or the diagonal in $\mathbb C^2$. 
\end{proposition}

\begin{proof}
Let the infinite sequence of roots be denoted as $(m_k!, n_k!)$ for $k=1, 2, \ldots$. Since $P$ is irreducible and has infinitely many integral solutions, by a standard descent argument (see, for instance, \cite[page 230]{Lan65}) we may assume that $P \in \Q[X, Y]$.  By clearing denominators, we may further assume that $P \in \Z[X, Y]$.

We first assume that the sequence $(m_k)_k$ is bounded. 
Then there exists $m_0 \geqslant 1$ such that $m_k=m_0$ for infinitely many $k$. By applying the division algorithm in $\mathbb C[X, Y]=\mathbb C[Y][X]$, we obtain that $P=(X-m_0!)Q+R$ for some $Q \in \mathbb C[X, Y]$ and $R \in \mathbb C[Y]$. Upon substituting $X=m_0!$ and $Y=n_k!$ for the infinitely many $k$ for which $m_k=m_0$, it follows that the polynomial $R$ has infinitely many roots and so must be zero. Since $P$ is irreducible, it follows that $C_P$ is the zero locus of $X-m_0!$ and so is a vertical line. A similar argument shows that if $(n_k)_k$ is bounded, then $C_P$ must be a horizontal line. 
\\
We now assume that both $(n_k)_k$ and $(m_k)_k$ are unbounded.  We will show that $n_k=m_k$ for infinitely many $k$. This will imply, by Bézout's theorem, that $C_P$ must be the diagonal. The discussion in Remark \ref{asymp} implies that, by passing to an infinite subsequence if necessary and relabeling indices,  there exists $r \in \Q$ and $c \in \mathbb C$ such that $n_k! \sim c (m_k!)^r$ as $k \to \infty$ . Note that we must necessarily have $c \in \R_{>0}$ and $r>0$.   
Taking logarithms yields 
\begin{equation} \label{pl}
\log(n_k!)= r \log(m_k!)+ \log c+o(1). 
\end{equation}

Applying Stirling's approximation  $\log(x!)=x \log x-x+O(\log x)$ to  \eqref{pl} implies
\begin{equation} \label{pl2}
n_k \log n_k - n_k = r m_k \log m_k - rm_k +O(\log n_k) +O(\log m_k). 
\end{equation}

We conclude in particular that $n_k \log n_k \sim r m_k \log m_k$. Applying logarithms again yields $\log n_k \sim \log m_k$ and so, by applying this back in $n_k \log n_k \sim r m_k \log m_k$, we obtain  $n_k \sim r m_k$ as $k \to \infty$.     

We now employ a $p$-adic argument to estimate the difference $n_k - r m_k$ along a suitable infinite subsequence. 
Let $P(X, Y)=\sum_{i, j} c_{i, j} X^i Y^j$ and for a fixed prime $p$, consider the $p$-adic valuation of each evaluated monomial $v_p(c_{i, j} (m_k!)^{i} (n_k!)^{j}) = v_p(c_{i, j})+i v_p(m_k!)+ j v_p(n_k!)$. 
Since $P(m_k!, n_k!)=0$ for each $k \geqslant 1$, the minimal $p$-adic valuation of at least two evaluated monomials must be equal. Thus, there is a specific pair of monomials, say with exponents $(i_1, j_1)$ and $(i_2, j_2)$, that have equal minimum evaluated valuation for an infinite subsequence. By relabeling our indices, we may again assume without loss of generality that this equality holds for all $k$ in our sequence.  By Legendre's formula, $v_p(x!)=\frac{x-s_p(x)}{p-1}$, where $s_p(x)$ is the sum of the digits in the base-$p$ expansion of $x$;  note that $s_p(x)=O(\log x)$. Inserting this into 
$v_p(c_{i_1, j_1} (m_k!)^{i_1} (n_k!)^{j_1}) 
= v_p(c_{i_2, j_2} (m_k!)^{i_2} (n_k!)^{j_2})$,  we obtain $i_1 m_k+j_1 n_k =i_2 m_k+j_2 n_k+O(\log n_k)+O(\log m_k)$. We now note that $\log n_k = O(\log m_k)$ and also that $j_1 \neq j_2$,  since otherwise $m_k=O(\log m_k)$ contradicting the fact that $(m_k)_k$ is unbounded.  A similar argument shows that $i_1 \neq i_2$. We now obtain that $n_k+ \frac{i_1-i_2}{j_1-j_2} m_k= O(\log m_k)$, which forces $r= \frac{i_2-i_1}{j_1-j_2}$ using again the fact that $n_k \sim r m_k$.  
In summary, if we let $E_k = n_k -r m_k$, then we have $E_k = O(\log m_k)$.

Substituting $n_k=rm_k+E_k$ and using $\log n_k=O(\log m_k)$ in \eqref{pl2} we obtain 
\begin{equation} \label{pl3}
(rm_k+E_k) \log(rm_k+E_k)-E_k - rm_k \log m_k = O(\log m_k). 
\end{equation}
Writing $\log(rm_k+E_k)= \log(rm_k)+ \log(1+ E_k/r m_k)= \log (rm_k)+O \left( \frac{E_k}{r m_k} \right)$ and using $E_k = O(\log m_k)$,  \eqref{pl3} simplifies to $r m_k \log r= O((\log m_k)^2)$ and hence to  
$\displaystyle{
r \log r= O \left( \frac{(\log m_k)^2}{m_k} \right).}$  Since the right hand side tends to zero as $k \to \infty$, it follows that $r=1$ and so $n_k! \sim c m_k!$ as $k \to \infty$,  \textit{i.e.},  $\lim_{k \to \infty} \frac{n_k!}{m_k!}=c$. Note that if $n_k> m_k$ for infinitely many $k$, then $\frac{n_k!}{m_k!} \geqslant m_k+1 \to \infty$, a contradiction. Similarly, if $n_k < m_k$ for infinitely many $k$, then $\frac{n_k!}{m_k!} \leqslant \frac{1}{n_k+1} \to 0$, a contradiction. 
Thus, we must have that $n_k=m_k$ for all $k$ sufficiently large as desired. 
\end{proof}

Thus, conditional on Conjecture \ref{bialg}, Theorem \ref{complexgt} and Proposition \ref{factorials} imply that the analogue of the Manin--Mumford conjecture holds for (products of) the Gamma function.  Subsequent to the completion of the present paper, the author was informed that Abhishek Oswal and Roy Zhao have obtained independent proofs of Proposition \ref{factorials}.

\subsection{Real analogue for $\Gamma$}  We call a point $(x, y)$ in $\mathbb R^2$ $\Gamma$-special if $(x, y) \in \overline \Q^2$ and $(x, y)= G( (a, b)) $ for some $(a, b) \in \mathbb R^2 \cap \overline \Q^2$; note that these correspond precisely to the $\Gamma$-special points of $\mathbb C$ under the natural identification $\mathbb R^2 \simeq \mathbb C$.  We call an irreducible real algebraic curve $W$ in $\mathbb R^2$ $\Gamma$-special if $W \supseteq G(\cV \cap (\mathbb R^2 \smallsetminus S))$ for a weakly bialgebraic curve $\cV \subseteq \mathbb R^2$. Conjecture \ref{bialg} implies that the only $\Gamma$-special points in $\mathbb R^2$ are of the form $(n!, 0)$ for $n \geq 1$. On the other hand, Theorem \ref{mainthm} and Proposition \ref{complete} imply that the $x$-axis is the only $\Gamma$-special real algebraic curve in $\mathbb R^2$. Conditional on Conjecture \ref{bialg}, we thus obtain the following real analogue of the Manin--Mumford conjecture for the Gamma function: \textit{suppose $W$ is an irreducible real algebraic curve in $\R^2$ containing infinitely many $\Gamma$-special points.  Then $W$ must be $\Gamma$-special}.

\end{document}